\documentclass[12pt,psamsfonts]{amsart}
\usepackage[utf8]{inputenc}

\usepackage{graphicx}
\usepackage{amsmath}
\usepackage{amssymb}
\usepackage{amscd}
\usepackage{amsfonts}
\usepackage{amsbsy}
\usepackage{graphicx}
\usepackage[dvips]{psfrag}
\usepackage{array}
\usepackage{color}
\usepackage{epsfig}
\usepackage{url}
\usepackage[hidelinks]{hyperref}
\usepackage{overpic}
\usepackage{epstopdf}
\usepackage{float}

\newcommand{\R}{\ensuremath{\mathbb{R}}}

\newcommand{\CB}{\ensuremath{\mathcal{B}}}

\newcommand{\CM}{\ensuremath{\mathcal{M}}}
\newcommand{\CN}{\ensuremath{\mathcal{N}}}

\newcommand{\CO}{\ensuremath{\mathcal{O}}}

\newcommand{\CU}{\ensuremath{\mathcal{U}}}

\newcommand{\ov}{\overline}

\newcommand{\G}{\Gamma}

\newcommand{\f}{\varphi}
\newcommand{\al}{\alpha}

\newcommand{\U}{\ensuremath{\mathcal{U}}}

\newcommand{\x}{\mathbf{x}}
\newcommand{\y}{\mathbf{y}}

\newcommand{\sgn}{\mathrm{sign}}
\newcommand{\de}{\delta}

\newcommand{\p}{\partial}

\newtheorem {theorem} {Theorem}
\newtheorem {definition} {Definition}

\newtheorem {lemma}  {Lemma}

\newtheorem {remark} {Remark}

\newtheorem {mtheorem} {Theorem}

\begin{document}
\renewcommand{\arraystretch}{1.5}

\title[Sliding Shilnikov Connection in Predator-Prey Model]
{Sliding Shilnikov Connection in Filippov-type Predator-Prey Model}

\author[T. Carvalho, D. D. Novaes and L. F. Gon\c calves]
{Tiago Carvalho$^1,$ Douglas D. Novaes$^2$ and L. F. Gon\c calves$^3$}

\address{$^1$ Departamento de Computa\c{c}\~{a}o e Matem\'{a}tica, Faculdade de Filosofia, Ci\^{e}ncias e Letras de Ribeir\~{a}o Preto,
              Universidade de S\~{a}o Paulo, Av. Bandeirantes, 3900, CEP 14040-901, Ribeir\~{a}o Preto, SP, Brazil.}\email{tiagocarvalho@usp.br}

\address{$^2$ Departamento de Matem\'{a}tica, Universidade
             Estadual de Campinas, Rua S\'{e}rgio Buarque de Holanda, 651, Cidade Universit\'{a}ria Zeferino Vaz, CEP 13083-859, Campinas, SP,
             Brazil.} \email{ddnovaes@unicamp.br}

\address{$^3$ Instituto de Bioci\^{e}ncias, Letras e Ci\^{e}ncias Exatas, Universidade Estadual Paulista (UNESP), Rua Crist\'{o}v\~{a}o Colombo, 2265,  CEP 15054-000, S\~{a}o Jos\'{e} do Rio Preto, SP, Brazil.}\email{luizfernandonandoo11@gmail.com}


\maketitle

\begin{abstract}
Recently, a piecewise smooth differential system was derived as a model of a 1 predator-2 prey interaction where the predator feeds adaptively on its preferred prey and an alternative prey.  In such a model, strong evidence of chaotic behavior was numerically found. Here, we revisit this model and  prove the existence of a Shilnikov sliding connection when the parameters are taken in a codimension one submanifold of the parameter space. As a consequence of this connection, we conclude, analytically, that the model behaves chaotically for an open region of the parameter space.
\end{abstract}


\section{Introduction}
In ecology, prey switching refers to a predator’s adaptive change of habitat or diet in response to prey abundance and has been observed in many predator species \cite{Allen,Gendron,Greenwood,Leeuwen}. A predator is said to be switching between prey species if  the number of attacks upon a species is disproportionately large when the species is abundant relative to other prey, and disproportionately small when the species is relatively rare \cite{Murdoch,Leeuwen}.
Switching in predators is often related to stabilizing mechanisms of prey populations and  is a possible explanation for coexistence \cite{Abrams,Krivan}. Intuitively,  predators tend to feed most heavily upon the most abundant prey species. As the prey species declines, the predator ``switches'' the major fraction of its attacks to another prey that has become the most abundant. This way, no prey population is drastically reduced nor becomes very abundant \cite{Murdoch}.

Using the principle of optimal foraging \cite{Stephens}, Piltz et al. \cite{Piltz} have introduced a model of a 1 predator-2 prey interaction as a piecewise differential system of kind
\begin{equation}\label{pss}
\dot \x=Z(\x)=F(\x)+\sgn(h(\x))G(\x),\\
\end{equation}
where $\x\in\R_{>0}^3,$ $h\colon\R_{>0}^3\rightarrow\R$ is linear, $F,G$ are smooth functions, and $\sgn:\R\setminus\{0\}\rightarrow\{-1,1\}$ stands for the sign function. Here, $\Sigma=h^{-1}(0)$ is called the switching manifold (or discontinuity manifold). The predator is assumed to instantaneously switch its food preference according to the availability of preys in the environment. This sudden change in the food preference of the predator induces discontinuities in the mathematical model used to describe such behavior.

The notion of trajectories for piecewise smooth differential systems of kind \eqref{pss} was stated by Filippov in \cite{F}. Nowadays, the differential equation \eqref{pss} is called Filippov system. It is worth mentioning the existence of a vast literature on Filippov systems modeling real phenomena in many other areas of applied science. For instance, see \cite{Rossa} for applications in  control theory, \cite{Brogliato,Dixon,Leine} in mechanical models, \cite{CarCrisPagTon-PhysicaD-2017,Kousaka} in electrical circuits, \cite{diBern-relay,Jac-To} in relay systems, \cite{GuptaGakkhar2016,ZhangTang2014} for biological models, \cite{RMCGslidingmode} for cancer modeling,  among others.  In all these applications, the discontinuity is due to an abrupt change in the differential system when some threshold is crossed.

Similar models were quite intensively studied during recent years in the context of the {\it ideal free distribution} \cite{Fretwell72,Fretwell1969},  which in essence is a hypothesis on how animals are distributed in a space constituted by habitats of varying suitability 
\cite{Smith72}.  In fact, one of the first to analyze equilibria of such models was Holt \cite{Holt1977} in 1977, who showed that the shared predator causes the so-called {\it apparent competition} between the two prey species,  that is, the presence of either species leads to a reduced population density for the other species at equilibrium. Filippov population models with prey switching were introduced by Colombo and Krivan \cite{Colombo93} in 1993  and extensively studied in subsequent articles (see, for instance, \cite{Krivan97,Kivan2006a,Kivan2008}).  In all these models, the evolution of the predator population $P(t)$ is ruled by piecewise smooth differential equations of the form:
\[
\dfrac{dP}{dt}=\left\{\begin{array} {l}
(e\, q_1 \beta_1 p_1 -m)P\quad\text{if} \quad \beta_1 p_1 -  \beta_2 p_2>0,\vspace{0.2cm}\\ 
(e\, q_2 \beta_2 p_2 -m)P\quad\text{if} \quad \beta_1 p_1 -  \beta_2 p_2<0,
\end{array}\right.
\]
where $p_1(t)$ and $p_2(t)$ represent the two prey populations.
This formulation assumes that predators' preference for prey is evolutionarily optimized. In the aforementioned models, $q_1\geq0$ and $q_2\geq0$ were interpreted as probabilities with which predators forage on preys $p_1$ and $p_2,$ respectively. This led to the natural trade-off in those models provided by $q_1 + q_2 = 1,$ which is a special case of the model introduced by Piltz et al. \cite{Piltz} that will be addressed in the present study (see Section \ref{mr}). 
Under this special condition, Boukal and Krivan \cite{Boukal99} proved the existence of a global attractor for the sliding dynamics by means of a suitable Lyapunov function.

Piltz et al. \cite{Piltz} found evidence that for a given choice of parameters their model exhibits chaotic behavior. Chaotic behavior can be understood as the existence of an invariant set for which the dynamics is transitive, sensitivity to initial conditions, and have dense periodic points (see, for instance, \cite{Devaney86,Meiss07,Wiggins90}, and Definitions \ref{definicao transitividade}, \ref{definicao sensibilidade}, and  \ref{definicao caotico} of Section \ref{prel}). For the biological model in question,  chaotic behavior means that for a given initial condition of population density of the species involved one cannot estimate (even vaguely) its long-term evolution. Hence, knowledge of chaotic behavior in a specific ecological model is of major importance, particularly for experimentalists who need to be aware of the potential implications of chaos for long-term predictions, and the fact that sustained ``irregular'' fluctuations may be due to chaos \cite{Hastings}.

In the smooth differential systems context, chaotic behavior may be tracked by studying the existence of objects previously known to be chaotic. This is the case of a {\it Shilnikov homoclinic orbit}, which is a trajectory connecting a hyperbolic saddle-focus equilibrium to itself, bi-asymptotically  (see \cite{S1,S2,T1}).

In the Filippov theory, pseudo-equilibria are special points on the switching manifold that must be distinguished and treated as typical singularities (see, for instance, \cite{F,Teixeira12}). As such, we can define the {\it sliding Shilnikov orbit} (see Definition \ref{defshil}), which is a trajectory in the Filippov terms connecting a hyperbolic pseudo saddle–focus to itself in an infinity time at least by one side, forward or backward. This object has been first considered in \cite{NT}, where some of their properties were studied. In particular, the existence of infinitely many sliding periodic solutions near a sliding Shilnikov orbit has been proven.

In \cite{NPV}, using the well-known theory of Bernoulli shifts, a full topological and ergodic description of the dynamics of Filippov systems near a sliding Shilnikov orbit $\Gamma$ was provided. In particular, it was established the existence of a set $\Lambda$ such that the restriction to $\Lambda$ of the first return map $\pi$ defined near $\G,$ is topologically conjugate to a Bernoulli shift with infinite topological entropy. This ensures that $\pi$ and, consequently, the flow are chaotic. As a consequence, given any natural number $m\geq 1,$ one can find infinitely many periodic points of the first return map with period $m$ and, consequently, infinitely many closed orbits near $\G.$ In addition, it was also provided that such a chaotic behavior persists, in some sense, under small perturbations. Namely, let $Z_{\al},$ $\al\in\R,$ be a smooth family of Filippov systems such that $Z_0$ has a sliding Shilnikov orbit. Then, for $|\al|$ sufficiently small, one can find a set $\Lambda_{\al}$ such that the restriction to $\Lambda_{\al}$ of the perturbed first return map $\pi_{\al}$ is topologically conjugate to a Bernoulli shift, eventually, of finite symbols. In this case, one can still find infinitely many closed orbits.

Our approach consists in finding a set of parameters for which the considered model admits a sliding Shilnikov orbit. This ensures, analytically, that the model behaves chaotically for parameters taken in a neighborhood of this set. 

The rest of this article is organized as follows. In Section \ref{prel},we introduce the essential theory on piecewise smooth vector fields and sliding Shilnikov orbits. The main result is stated in Section \ref{mr} while its proof as well as the analysis of the model are provided in Section \ref{proofmr}. In Section \ref{simul}, we perform a numerical simulation to exhibit a sliding Shilnikov connection. Some conclusion remarks and further directions are provided in Section \ref{sec:conc}.

\section{Preliminary concepts and known results}\label{prel}

This section serves to present the essential theory on Piecewise Smooth Vector Fields. The concept of \textit{Sliding Shilnikov orbits} shall also be defined and some known results regarding the chaotic behavior near a sliding Shilnikov orbit will be provided.

A piecewise smooth vector field on $\R^3$ is a pair of C$^r$-vector fields $X$ and  $Y,$ where $X$ and  $Y$  are restricted to regions of $\R^3$  separated by a smooth  codimension one manifold $\Sigma.$  The \textit{switching manifold} $\Sigma$ is obtained considering $\Sigma=h^{-1}(0),$ where $h$ is a differentiable function having $0$ as a regular value. Note that $\Sigma$ is the separating boundary of the regions $\Sigma^+=\{\x\in \R^3 \, | \, h(\x) >0\}$ and $\Sigma^-=\{\x \in \R^3 \,| \, h(\x)< 0\}.$ Thus, a piecewise smooth vector field is provided by:
\begin{equation}\label{omega}
Z(\x)=\left\{\begin{array}{l}
X(\x),\quad\textrm{if}\quad h(\x)> 0,\vspace{0.1cm}\\
Y(\x),\quad\textrm{if}\quad h(\x)< 0,
\end{array}\right.
\end{equation}
where $\x \in \R^3.$ System \eqref{omega} is denoted by $Z=(X,Y).$ Call $\chi^r$ the space of piecewise smooth vector fields. We endow $\chi^r$ with the product topology.

In order to establish a definition for the trajectories of $Z$ and investigate its behavior, we need a criterion for the transition of the orbits between $\Sigma^+$ and $\Sigma^-$ across $\Sigma.$ The contact between the vector field $X$ (or $Y$) and the switching manifold $\Sigma$ is characterized by the expression $$Xh(\x)=\left\langle \nabla h(\x), X(\x)\right\rangle$$ and, for $i\geq 2,$  $X^i h(\x)=\left\langle \nabla X^{i-1} h(\x), X(\x)\right\rangle,$ where $\langle . , . \rangle$ is the usual inner product in $\R^3.$ The basic results of differential equations in this context were stated by Filippov \cite{F}. We can divide the switching manifold in the following sets.

\begin{definition} Consider a piecewise smooth vector field $Z=(X,Y).$ 
	\begin{itemize}
		\item [$(i)$] A set $\Sigma^c$ is called a crossing set if $Xh(\x)\cdot Yh(\x) > 0$ for all $\x\in\Sigma^c$;
		\item [$(ii)$] A set $\Sigma^e$ is called a escaping set if $Xh(\x)>0$ and $Yh(\x) < 0$ for all $\x\in\Sigma^e$;
		\item [$(iii)$] A set $\Sigma^s$ is called a sliding set if $Xh(\x)<0$ and $Yh(\x) > 0$ for all $\x\in\Sigma^s.$
	\end{itemize}
\end{definition}

The {\it escaping} $\Sigma^e$ or {\it sliding} $\Sigma^s$ regions are respectively defined on points of $\Sigma$ where both vector fields $X$ and $Y$ simultaneously point outwards or inwards from $\Sigma$ while the interior of its complement in $\Sigma$ defines the {\it crossing region} $\Sigma^c$ (see Fig. \ref{regioescostdeslizeescape}). The complementary of the union of those regions is the set formed by the {\it tangency} points between $X$ or $Y$ with $\Sigma$ (see, for instance, \cite{Teixeira12}). 

\begin{figure}[h!]
	\begin{center}
		\includegraphics[width=2.5in]{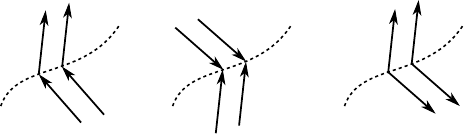}
	\end{center}
	\caption{Crossing $\Sigma^c$, sliding $\Sigma^s,$ and escaping $\Sigma^e$ regions in Filippov systems, respectively.}\label{regioescostdeslizeescape}
\end{figure}

\begin{definition}
	A point $\x\in\Sigma$ is called a {\it tangency} point of $X$ (resp. $Y$)  if it satisfies $Xh(\x)=0$ (resp. $Yh(\x)=0$). A tangency point is called a  \textit{fold} point of $X$ if $X^2h(\x)\neq0.$ Moreover, $\x\in\Sigma$ is a visible (resp. invisible) fold point of $X$ if $X^2h(\x)>0$ (resp.  $X^2h(\x)<0$).  A tangency point is called a  \textit{cusp} point of $X$ if $X^2h(\x)=0$ and $X^3h(\x)\neq0.$
\end{definition}

In order to define a trajectory of a PSVF passing through a crossing point, it is enough to concatenate the trajectories of $X$ and $Y$ by that point. However, in the sliding and escaping sets we need to define an auxiliary vector field. Thus, we consider the Filippov's convention (see \cite{F}) and a new vector field is defined on $\Sigma^s \cup \Sigma^e.$

\begin{definition} Given a point $\x \in \Sigma^s\cup\Sigma^e,$ we define the \textit{sliding vector field} at $\x$ as the vector field $Z^{s}(\x)=\y-\x,$ with $\y$ being the point of the segment joining $\x+X(\x)$ and $\x+Y(\x)$ such that $\y-\x$ is tangent to $\Sigma.$
\end{definition}
The sliding vector field is provided by the expression
\begin{equation}\label{slisys}
Z^s(\x)=\dfrac{Y h(\x) X(\x)-X h(\x) Y(\x)}{Y h(\x)- Xh(\x)} \,\,\mbox{, } \x \in \Sigma^s\cup\Sigma^e.
\end{equation}
In this scenario, the trajectories $\Gamma_{Z}(t,q)$ of $Z$ are considered a concatenation of trajectories of $X,$ $Y$ and $Z^s.$

The points $\x \in \Sigma^s\cup\Sigma^e$ such that $Z^s(\x)=0$ are called \textit{pseudo equilibrium of $Z$}. A pseudo-equilibrium is called  {\it hyperbolic pseudo-equilibrium} when it is a hyperbolic critical point of $Z^s.$ In particular, if $\x^*\in\Sigma^s$ (resp. $\x^*\in\Sigma^e$) is an unstable (resp. stable) hyperbolic focus of $Z^s,$ then we call $\x^*$ a {\it hyperbolic saddle-focus pseudo-equilibrium} or just {\it hyperbolic pseudo saddle-focus}.

\begin{definition}\label{defshil}
	Let $Z=(X,Y)$ be a piecewise continuous vector field having a hyperbolic pseudo saddle-focus $p\in \Sigma^{s}$ $($resp. $p\in \Sigma^{e}),$ and let  $q\in\p\Sigma^s$ $($resp. $q\in\p\Sigma^e)$ be a visible fold point of the vector field $X$ such that
	\begin{itemize}
		\item[$(i)$] the orbit passing through $q$ following the sliding vector field $Z^s$ converges to $p$ backward in time $($resp. forward in time$)$;
		
		\item[$(ii)$] the orbit starting at $q$ and following the vector field $X$ spends a time $t_0>0$ $($resp. $t_0<0)$ to reach $p.$
	\end{itemize}
	Thus, through $p$ and $q$ a sliding loop $\G$ is easily characterized. We call $\G$ a {\bf sliding Shilnikov orbit} $($see Fig. \ref{slidingshil}$).$
\end{definition}

\begin{figure}[h]
	\begin{center}
		\begin{overpic}[width=6.7cm]{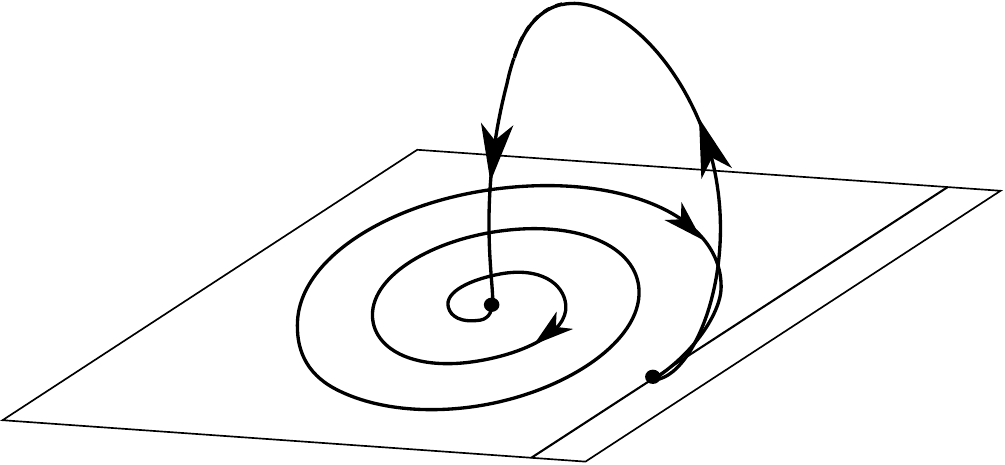}
			\put(66,42){$\G$}
			\put(50.5,15){$p$}
			\put(63.5,5.5){$q$}
			\put(10,6){$\Sigma^s$}
			\put(92,29){$\p \Sigma^s$}
		\end{overpic}
	\end{center}
	
	\bigskip
	
	\caption{ The point $p\in \Sigma^s$ is a hyperbolic pseudo saddle–focus. The trajectory $\G,$ called Shilnikov sliding orbit, connects $p$ to itself passing through the point $q\in\p \Sigma^s,$ the frontier of $\Sigma^s.$ Notice that the flow leaving $q$ reaches the point $p$ in a finite positive time, and approaches backwards to $p,$ asymptotically. }\label{slidingshil}
\end{figure}
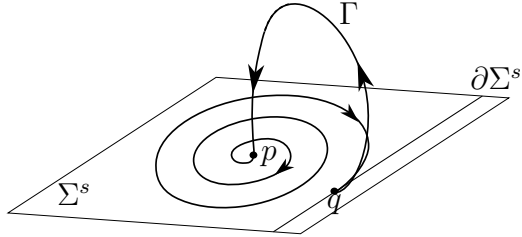

Chaotic behavior of a dynamical system is usually understood as the existence of an invariant set for which the dynamics is transitive, sensitivity to initial conditions, and have dense periodic points (see, for instance, \cite{Devaney86,Meiss07,Wiggins90}). In this way, let us consider the following definitions.

\begin{definition}\label{definicao transitividade}
	The Filippov system \eqref{omega} is \textbf{topologically transitive} on an invariant set $W$ if for every pair of nonempty open sets $U$ and $V$ in $W$ there exist $\x\in U,$  a  trajectory of $Z,$ $\Gamma_{Z}(t,\x),$ and $t_{0}>0$ such that $\Gamma_{Z}(t_{0},\x) \in V.$
\end{definition}

\begin{definition}\label{definicao sensibilidade}
	The Filippov system \eqref{omega} exhibits \textbf{sensitive dependence} on a compact invariant set $W$ if there is a fixed $r>0$ satisfying $r < \textrm{diam}(W)/2$ such that for each $\x\in W$ and $\varepsilon>0$ there exist $\y\in B_{\varepsilon}(x)\cap W$ and positive global trajectories $\Gamma_{\x}^{+}$ and $\Gamma_{\y}^{+}$ passing through $\x$ and $\y,$ respectively, satisfying
	$$
	d(\Gamma_{\x}^{+}(t_0),\Gamma_{\y}^{+}(t_0))>r,
	$$
	where $t_0 \in \R$ is positive, $d$ is the Euclidean distance and $diam(W)$ is the diameter of $W,$ i.e., the largest distance between two elements of $W.$
\end{definition}

\begin{definition}\label{definicao caotico}
	The Filippov system \eqref{omega} is chaotic on a compact invariant set $W$ if it is topologically transitive, exhibits sensitive dependence on $W,$ and has dense periodic orbits in $W.$
\end{definition}

The next theorem ensures that a Filippov system presenting a sliding Shilnikov connection is, in fact, chaotic. The proof of Theorem \ref{theo:topological} is performed in \cite{NPV}.

\begin{theorem}[\cite{NPV}]\label{theo:topological}
	Let $Z=(X,Y)$ be provided by \eqref{omega}. Assume that the Filippov system $Z$ admits a sliding Shilnikov orbit. Then, there exists a neighborhood $\CU\subset\chi^r$ of $Z$ such that each $\widetilde{Z}\in\CU$ admits an invariant compact set $\Lambda_{\widetilde{Z}},$ in which $\widetilde{Z}$ is chaotic.
\end{theorem}

\section{Main results}\label{mr}
Ciliates are eukaryotic single cells that belong to the protist kingdom. They occur in aquatic environments and feed on small phytoplankton, constituting a relevant link between levels of marine and freshwater food webs (see \cite{TG1}). Coexistence of species in a shared environment may arise from ecological trade-offs (see \cite{KneitelChase}), which appear in many situations in ecology. Lake Constance is a freshwater lake situated on the German-Swiss-Austrian border that has been under scientific investigation for decades,  and a substantial amount of data on the biomass of several phytoplankton and zooplankton species is available (see \cite{Gaedke,TG1,TG2}). Based on the available data, Piltz et al.  \cite{Piltz} derive the following piecewise smooth model for a 1 predator-2 prey interaction where the predator feeds adaptively on its preferred prey and an alternative prey:

\begin{equation}\label{sistema inicial}
\dot\x=\left\{\begin{array}{ll}
\left(
\begin{array}{c}
(r_1 - \beta_1 P)p_1\\
r_2 p_2\\
(e\, q_1 \beta_1 p_1 -m)P
\end{array}
\right)&\textrm{if}\quad H(p_1,p_2,P)>0,\vspace{0.3cm}\\

\left(
\begin{array}{c}
r_1 p_1\\
(r_2 - \beta_2 P)p_2\\
(e\, q_2 \beta_2 p_2 -m)P
\end{array}
\right)&\textrm{if}\quad H(p_1,p_2,P)<0,
\end{array}\right.
\end{equation}

where $\dot\x=\big(\dot p_1,\dot p_2,\dot P\big)^T,$ $(p_1,p_2,P)\in\R_{\geq0}^3$ and $$H(p_1,p_2,P)=\beta_1 p_1 - a_q \beta_2 p_2.$$ The plane $S=H^{-1}(0)$ is the switching manifold of the piecewise differential system \eqref{sistema inicial}. The variables of the model \eqref{sistema inicial}, $P, p_1,$ and $p_2$ represent the density of the predator population, preferred prey, and alternative prey, respectively.
Regarding the parameters, $q_i\geq0$ represent the preference for prey $i,$ $i\in\{1,2\},$ and $a_q > 0$ is the slope of the preference trade-off. The intercept of the preference trade-off $b_q=q_2-a_q q_1$ is assumed to satisfy $b_q\geq0.$ In addition, $e>0$ is the proportion of predation that goes into predator
growth, $\beta_1>0$ and $\beta_2>0$ are, respectively, the death rates of the preferred and alternative prey due to predation. Finally, $m>0$ is the predator per capita death rate per day and $r_1>r_2>0$ are the per capita growth rates of the preferred and alternative prey, respectively.

The above constraints imply that the parameters of the Filippov system \eqref{sistema inicial} lie in a subset of the Euclidean space $\R^9,$ namely $$\eta=(r_1,r_2,a_q,q_1,q_2,\beta_1,\beta_2,m,e)\in\CM=R\times Q\times\R_{> 0}^4,$$ where $R=\{(r_1,r_2)\in\R^2_{> 0}:\, r_1>r_2\}$ and $Q=\{(a_q,q_1,q_2)\in\R_{> 0}\times\R_{\geq 0}^2:\,q_2\geq a_q q_1\}.$ The set $\CM$ is called {\it space of parameters} which is a $9$-dimensional submanifold of $\R^9$ with boundary and corner.

In what follows, we state the main result of this paper that guarantees that exist parameters such that the prey switching model \eqref{sistema inicial} possesses a sliding Shilnikov orbit and exhibits chaos. This is proved in Section \ref{proofthm2}.
\begin{mtheorem}\label{maintheorem}
	There exists a codimension one submanifold $\CN$ of $\CM$ such that the Filippov system \eqref{sistema inicial} possesses a sliding Shilnikov orbit whenever $\eta\in\CN.$ Moreover, there exists a neighborhood $\CU\subset \CM$ of $\CN$ such that the Filippov system \eqref{sistema inicial} behaves chaotically whenever $\eta\in \CU.$
\end{mtheorem}

\section{Proof of the main result}\label{proofmr}
Consider the differential piecewise differential system \eqref{sistema inicial}. In order to eliminate the dependence of the switching manifold on the parameters, let us consider the change of variables $x= \beta_1 p_1,$ $y= a_q \beta_2 p_2$ and  $z = \beta_1 P.$ In these new variables, \eqref{sistema inicial} writes

\begin{equation}\label{sistema inicial sem beta}
\big(\dot x,\dot y,\dot z\big)^T=\left\{\begin{array}{ll}
X(x,y,z)&\quad\textrm{if}\quad h(x,y,z)>0;\vspace{0.2cm}\\

Y(x,y,z)&\quad\textrm{if}\quad h(x,y,z)<0,
\end{array}\right.
\end{equation}
where $X(x,y,z)=\left((r_1 - z)x,r_2 y,(e q_1 x -m)z\right),$ 

$$Y(x,y,z)=\left(r_1 x,\left(r_2 - \dfrac{\beta_2 }{\beta_1}z\right)y,\left(\dfrac{e q_2 }{a_q}y -m\right)z\right),$$ 
$(x,y,z)\in\R_{\geq0}^3$ and $h(x,y,z)=x-y.$ Now, the switching manifold is provided by $\Sigma=h^{-1}(0)=\{(x,x,z):\,x\geq 0,\,z\geq0\}.$

\subsection{Dynamics of $X$ and $Y$ and their contacts with $\Sigma$}\label{sec dinamica no contato}
Notice that the plane $\Pi_y=\{y=0\}$ is invariant through the flow of $X.$ The restriction of $X$ onto the plane $\Pi_y$ reads
\begin{equation}\label{eq projecao}
\overline{X}(x,z)=\left(
\begin{array}{c}
(r_1 -  z)x\\
(e q_1  x -m)z
\end{array}
\right).\end{equation}
Moreover, the projection of each orbit of $X$ into the plane $\Pi_y$ coincides with an orbit of $\overline{X}.$ Indeed, the subsystems $(\dot{x}, \dot{z})$ and $\dot{y}$ are uncoupled. The equilibria of $\overline{X}$ are $E_1=(0,0)$ and $E_2=(m/(e q_1),r_1).$ The equilibrium $E_1$ is a saddle with eigenvectors $(1,0)$ and $(0,1)$ associated to the eigenvalues $r_1$ and $-m,$ respectively.  The equilibrium  $E_2$ has pure imaginary eigenvalues, namely $ \pm i \sqrt{m r_1} .$ Furthermore, $\overline{X}$ is a Lotka-Volterra system which has the following first integral:
\begin{equation}\label{first}
F(x,z)=  -m -r_1+e q_1 x + z - m \log\left(\dfrac{e q_1\,x}{m}\right) - r_1 \log\left(\dfrac{z}{r_1}\right).
\end{equation}
It implies that the equilibrium $E_2$ is a center (see Fig. \ref{Fig porj y=0}).
\begin{figure}[h]
	\begin{center}
		\begin{overpic}
			[width=7cm]
			{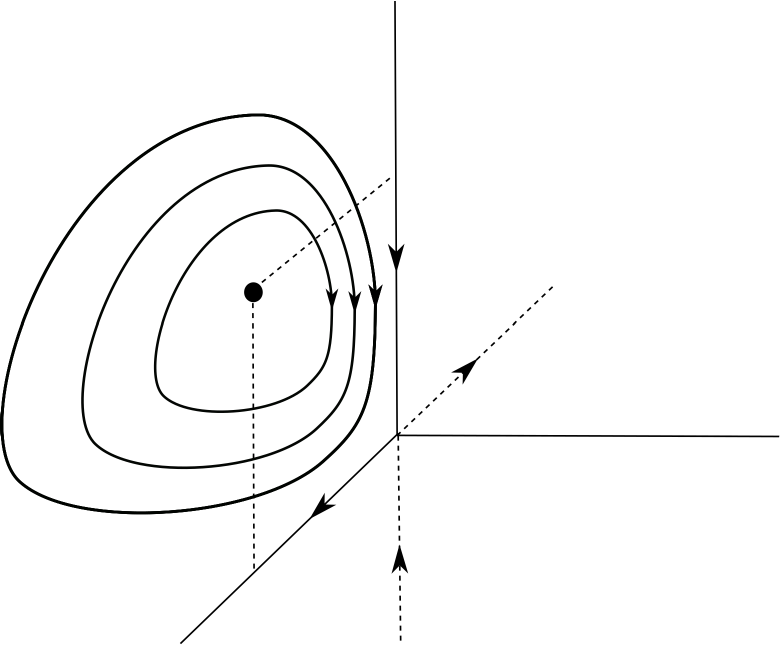}
			\put(19,-3){$x$}	\put(101,26){$y$}	\put(49,84){$z$}	\put(52,60){$r_1$}	\put(30,3){$\frac{m}{e q_1}$}
		\end{overpic}
	\end{center}
	\caption{Projection of the vector field $X$ onto the plane $\Pi_y.$}\label{Fig porj y=0}
\end{figure}
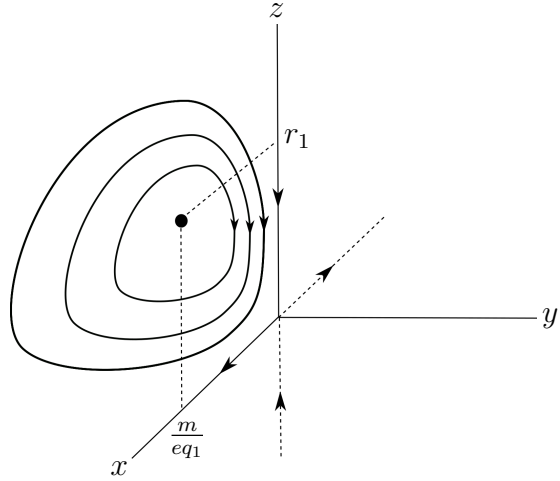
Furthermore, since \begin{equation}\label{eq dinamica direcao y}y(t)=y_0 \exp(r_2 t),\end{equation} the dynamics on the $y$-direction is unbounded increasing and the $X$-trajectories spiral from $\Pi_y$ toward $\Sigma,$  crossing $\Sigma.$ The trajectories of $X,$ on the domain $\R_{\geq0}^3,$ lie on cylinders around the straight line $$\ell=\{(m/(e q_1),y,r_1) \, | \, y\geq0\}.$$ See Fig. \ref{Fig cilindro invariante}.

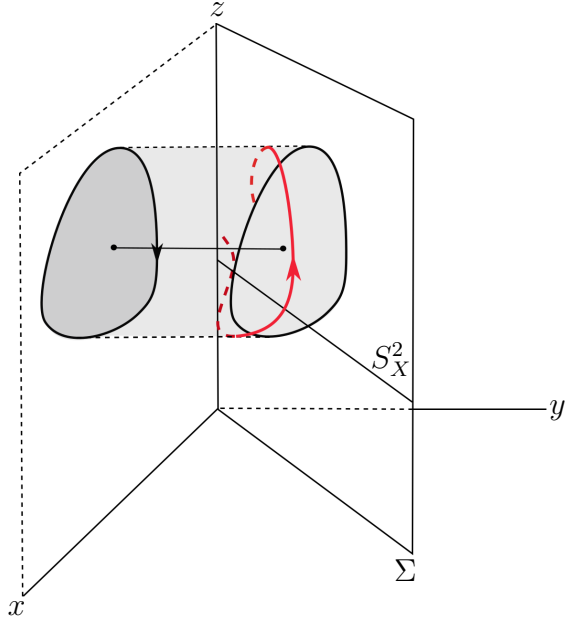
\begin{figure}[h]
	\begin{center}
		\begin{overpic}
			[width=7cm]
			{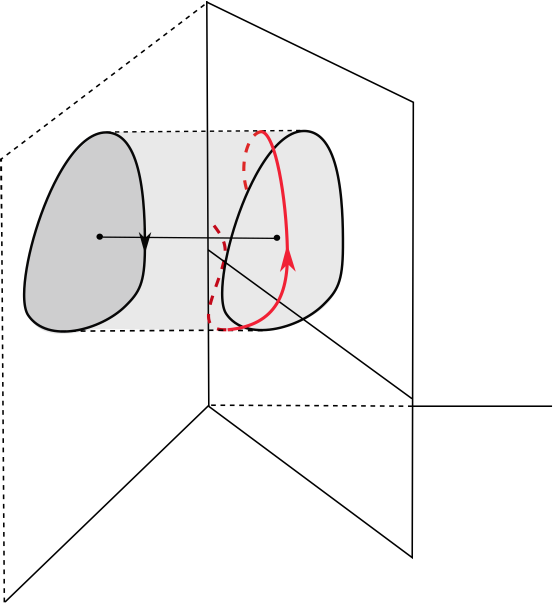}
			\put(-2,-3){$x$}	\put(92,32){$y$}	\put(33,101){$z$}	\put(65,3){$\Sigma$} \put(61,40){$S^{2}_{X}$}		
			
		\end{overpic}
	\end{center}
	\caption{Invariant cylinder of the vector field $X.$}\label{Fig cilindro invariante}
\end{figure}

The next step is to study the contacts of the vectors fields $X$ and $Y$ with the switching manifold $\Sigma.$ Let us consider $p=(x,x,z)\in\Sigma.$ Computing the Lie derivatives $Xh(p)$ and $Yh(p)$ we get:
\[	
Xh(p) = (r_1 - r_2 -  z) x\,\, \text{and}\,\, Yh(p) = \left(r_1 - r_2 + \dfrac{\beta_2 z}{\beta_1}\right) x.
\]
By solving equation $Xh(p)=0$ we conclude that the contacts between the vector field $X$ and the switching manifold $\Sigma$ occur at $S_X = S^{1}_X \cup S^{2}_X,$ where $S^{1}_X=\{(0,0,z):\,z\geq0\}$ and $S^{2}_X=\{(x,x,r_1-r_2):\,x>0\}.$ Analogously, solving the equation $Yh(p)=0,$ we conclude that the contacts between $Y$ and $\Sigma$ occurs at $S_Y = S^{1}_Y \cup S^{2}_Y$ where $S^{1}_Y=S^{1}_X$ and $S^{2}_Y=\{(x,x,-r_1+r_2)\}.$ The switching manifold $\Sigma$ is then partitioned into two open regions, namely \textit{sliding region} $\Sigma^s=\{(x,x,z)\in \Sigma \, | \, z> r_1-r_2 \}$ and  \textit{crossing region}  $\Sigma^c=\{(x,x,z)\in \Sigma\, | \, 0<z< r_1-r_2\}.$

Notice that the tangency line $S^{2}_X$ is the boundary of the sliding region $\Sigma^s.$ From Definition \eqref{defshil}, $S^{2}_X$ will play an important role in finding a sliding Shilnikov orbit. In order to determine the kind of contact between $X$ and $\Sigma$ occurring on $S^{2}_X,$ we compute the second Lie derivative $X^{2}h.$ Accordingly, let $\overline{p}=(x,x,r_1-r_2)\in S^{2}_X,$ so
$X^{2}h(\overline{p}) =  (r_1 - r_2) (m- e q_1 x) x.$
Solving the equation, $X^2h(\ov p)=0,$ we obtain two solutions, namely
\[
\overline{p}_1=\left(0,0,r_1 - r_2\right)\,\,\, \text{and} \,\,\, c=\left(\dfrac{m}{e q_1},\dfrac{m}{e q_1},r_1 - r_2\right).
\]
Moreover, we have that $X^{2}h(x,x,r_1 - r_2)>0$ for $0<x<m/(e q_1).$ Thus, $S^v_X=\{(x,x,r_1-r_2) \in S^{2}_X \, | \, 0<x<m/(e q_1)\}$ is a curve of {\it visible fold points} of $X$ and, therefore, the local trajectories of $X$ remain at the region where $X$ is defined (i.e. $h(x,y,z)>0$), before and after the tangential contact with $S^{v}_X.$  It is worth noting that $c\in \R^3_{\geq0}$ is a contact of cusp type.

In order to state the main result of this subsection (Lemma \ref{lemma1}), we introduce the following new parameters
\begin{equation}\label{par1}
\phi=r_1-r_2\quad \text{and}\quad\tau=\dfrac{m}{e\, q_1}.
\end{equation}
Solving the above relations for $r_1$ and $m$ we get $S^2_X=\{(x,x,\phi) \in S^{2}_X \, | \, x\geq0\}$ and  $c=\left(\tau,\tau,\phi \right).$

\begin{lemma}\label{lemma1}  For each $x_0\in(0,\tau),$ the forward trajectory of $X$ passing through $(x_0,x_0,r_1-r_2)$ intersects the switching manifold $\Sigma$ transversally at a point called $\mu(x_0)=(u(x_0),u(x_0),v(x_0)).$ In other words, the saturation of $S^v_X$ through the forward flow of $X$ intersects $\Sigma$ transversally  in the curve $\{\mu(x_0):\,0<x_0<\tau\}.$
	Moreover, the following statements hold:
	\begin{itemize}
		\item[i)] for $x_0<\tau$ sufficiently close to $\tau$ we have
		\begin{equation*}
		\begin{array}{l}
		u(x_0)=\tau-2(x_0-\tau)+\CO(x_0-\tau)^2,\\
		v(x_0)=r_1-r_2+\CO(x_0-\tau)^2;
		\end{array}
		\end{equation*}
		
		\item[ii)] and given $x_0\in(0,\tau),$ for $r_2>0$ sufficiently small we have
		\begin{equation*}
		\begin{array}{l}
		u(x_0)=x_0+\CO(r_2),\\
		v(x_0)=r_1+\sqrt{2r_1T(x_0)(m-e q_1 x_0)}\sqrt{r_2}+\CO\left(r_2^{3/2}\right),
		\end{array}
		\end{equation*}
	\end{itemize}
	where $T(x_0)$ is the period,  for $r_2=0,$ of the solution $\left(x(t,x_0;r_2), z(t,x_0;r_2)\right).$
\end{lemma}

\begin{proof}
	Take $(x_0,x_0,\phi)\in S_x^v,$ such that $0<x_0<\tau.$ The parameter $r_2$ will play an important role in this proof, so we shall make it explicit, as follows. Let us consider $\psi(t,x_0;r_2)=(x(t,x_0;r_2),y(t,x_0;r_2),z(t,x_0;r_2))$ the solution of $X$ such that $\psi(0,x_0;r_2)=(x_0,x_0,\phi).$ Notice that
	\[
	\begin{array}{l}
	\dfrac{\p x}{\p t}(0,x_0;r_2)=r_2 x_0,\vspace{0.2cm}\\
	\dfrac{\p^2 x}{\p t^2}(0,x_0;r_2)=r_2^2x_0+e q_1 x_0 (\tau-x_0)\phi,\vspace{0.2cm}\\
	\dfrac{\p y}{\p t}(0,x_0;r_2)=r_2 x_0,\quad\text{and}\quad \dfrac{\p^2 y}{\p t^2}(0,x_0;r_2)=r_2^2 x_0.
	\end{array}
	\]
	Since
	\[
	\begin{array}{l}
	x(0,x_0;r_2)=y(0,x_0;r_2)=x_0,\vspace{0.2cm}\\ 
	\dfrac{\p x}{\p t}(0,x_0;r_2)=\dfrac{\p y}{\p t}(0,x_0;r_2),\quad \text{and}\vspace{0.2cm}\\
	\dfrac{\p^2 x}{\p t^2}(0,x_0;r_2)>\dfrac{\p^2 y}{\p t^2}(0,x_0;r_2),
	\end{array}
	\]
	we conclude that
	$y(t,x_0;r_2)<x(t,x_0;r_2)$ for $t>0$ sufficiently small. However, $x(t,x_0;r_2)$ is bounded and $y(t,x_0;r_2)$ is unbounded increasing, therefore there exists a first positive time $t_1(x_0;r_2)>0$ such that	
	\begin{eqnarray}\label{xt1}
	x(t_1(x_0;r_2),x_0;r_2)&=&y(t_1(x_0;r_2),x_0;r_2)\\\nonumber
	&=&x_0\exp(r_2 t_1(x_0;r_2)).
	\end{eqnarray}
	It means that the trajectory of $X$ passing tangentially by each $(x_0,x_0,\phi)\in S_X^v,$ for $0<x_0<\tau,$ transversally reaches the switching manifold $\Sigma$ at the point $\psi(t_1(x_0;r_2),x_0;r_2).$ Accordingly, for $0< x_0<\tau,$ we can define $\mu(x_0)=\psi(t_1(x_0;r_2),x_0;r_2),$
	\begin{equation}\label{uv}
	\begin{array}{l}
	u(x_0)=x(t_1(x_0;r_2),x_0;r_2)=x_0\exp(r_2 t_1(x_0;r_2)),\,\,\text{and}\\
	
	v(x_0)=z(t_1(x_0;r_2),x_0;r_2).
	\end{array}
	\end{equation}
	This concludes the proof of the the first part of the lemma.
	Now, let us prove that the Taylor series of $u(x_0)$ around $x_0=\tau$ reads
	\begin{equation}\label{taylor_u}
	u(x_0)=\tau-2(x_0-\tau)+\CO_2(x_0-\tau)^2.
	\end{equation}
	Notice that the difference $x(t,x_0;r_2)-x_0\exp(r_2 t)$ around $t=0$ reads
$$\scriptstyle	-\dfrac{\scriptstyle e q_1 \phi x_0(x_0-\tau)}{\scriptstyle 2}t^2-\dfrac{\scriptstyle e q_1 \phi x_0(r_2(4x_0-3\tau)+e q_1(x_0-\tau)^2)}{\scriptstyle 6}t^3+\CO_4(t).$$
	Therefore, the function
	\[
	\Delta(t,x_0):=\dfrac{x(t,x_0;r_2)-x_0\exp(r_2 t)}{t^2}
	\]
	is well defined and, around $t=0,$ reads
	\[
	\begin{array}{rl}
	\Delta(t,x_0)=&-\dfrac{e q_1 \phi x_0(x_0-\tau)}{2}-\vspace{0.2cm}\\
	&\dfrac{e q_1 \phi x_0(r_2(4x_0-3\tau)+e q_1(x_0-\tau)^2)}{6}t+\cdots
	\end{array}
	\]
	In order to apply the Implicit Function Theorem, we compute
	\[
	\Delta(0,\tau)=0,\quad\dfrac{\p \Delta}{\p t}(0,\tau)=-\dfrac{e q_1 r_2 \tau^2\phi}{6}\neq0,\]
	\[
	\text{and}\quad \dfrac{\p \Delta}{\p x_0}(0,\tau)=-\dfrac{e q_1\tau\phi}{2}.
	\]
	Therefore, we find a unique function $t_2(x_0)$ such that
	\begin{equation}\label{t2}
	t_2(\tau)=0, \quad t_2'(\tau)=-\dfrac{\dfrac{\p \Delta}{\p x_0}(0,\tau)}{\dfrac{\p \Delta}{\p t}(0,\tau)}=-\dfrac{3}{r_2\tau}.
	\end{equation}
	From the uniqueness of $t_2$ we conclude that, for $x_0$ sufficiently close to $\tau,$ $t_1(x_0;r_2)=t_2(x_0).$ Thus, using \eqref{t2}, $u(x_0)=x_0\exp(r_2 t_1(x_0))$ can be expanded around $x_0=\tau$ in order to get \eqref{taylor_u}.
	
	Finally, we shall prove that given $x_0^*\in(0,\tau)$ there exists a neighborhood $\U$ of $x_0^*$ and $r_2^*>0$ such that $u(x_0)<\tau$ and $v(x_0)>r_1$ for every $(x_0,r_2)\in \U\times(0,r_2^*].$ Indeed, consider the function
	\[
	\de(t,r_2)=x(t,x_0;r_2)-y(t,x_0;r_2)=x(t,x_0;r_2)-x_0 e^{r_2 t}.
	\]
	We know that $\big(x(t,x_0;r_2),z(t,x_0;r_2)\big)$ is periodic in the variable $t$  (see Fig. \ref{Fig porj y=0}). In fact, this is the solution of the Lotka-Volterra system \eqref{eq projecao} with initial condition $(x_0,r_1-r_2)$ and, therefore, satisfies \eqref{first}
	\begin{equation}\label{Fr2}
	F\big(x(t,x_0;r_2),z(t,x_0;r_2)\big)=F(x_0,r_1-r_2),
	\end{equation}
	for every $r_1>r_2>0,$ $0<x_0<\tau,$ and $t$ on its interval of definition. Thus, for $r_2=0,$ denote by $T(x_0)>0$ the period of the solution $\big(x(t,x_0;0),z(t,x_0;0)\big),$ that is, $\big(x(T(x_0),x_0;0),z(T(x_0),x_0;0)\big)=(x_0,r_1).$ Therefore, $\de(T(x_0),0)=0.$ We shall see that there is a saddle-node bifurcation occurring at $t=T(x_0)$ for the critical value of the parameter $r_2=0.$ Computing the derivative in the variable $r_2$ of \eqref{Fr2} at $t=T(x_0)$ and $r_2=0$ we get
	\[
	\dfrac{\p x}{\p r_2}(T(x_0),x_0,0)=0.
	\]
	Thus, we get
	\[
	\dfrac{\p \de}{\p t}(T(x_0),0)=0,\]
	\[\dfrac{\p^2 \de}{\p t^2}(T(x_0),0)=r_1(m-e q_1 x_0)x_0>0,
	\]
	and
	\begin{equation}\label{dr2}
	\dfrac{\p \de}{\p r_2}(T(x_0),0)=-x_0T(x_0)<0.
	\end{equation}
	This implies the existence of a saddle-node bifurcation. In order to conclude this proof, we shall  explicitly compute the solutions bifurcating from $t=T(x_0).$ From $\eqref{dr2},$ applying the Implicit Function Theorem, we get the existence of neighborhoods $I_1$ and $V_1$ of $T(x_0)$ and $0,$ respectively, and a unique differentiable function $\rho:I_1\rightarrow V_1$ such that $\de(t,\rho(t))=0$ for every $t\in I_1.$ Moreover,
	\[
	\rho(T(x_0))=\rho'(T(x_0))=0\]
	\[\text{ and } \rho''(T(x_0))=\dfrac{r_1(m-e q_1 x_0)}{2T(x_0)}.
	\]
	Notice that we are taking
	\begin{equation}\label{r2}
	r_2=\rho(t)=\dfrac{r_1(m-e q_1 x_0)}{2T(x_0)}(t-T(x_0))^2+\CO(t-T(x_0))^3.
	\end{equation}
	Proceeding with the change $s=(t-T(x_0))^2,$ equation \eqref{r2} is equivalent to
	\[
	r_2=\dfrac{r_1(m-e q_1 x_0)}{2T(x_0)}s+\CO(|s|^{3/2}).
	\]
	It is easy to see that the above equation can be inverted using the Inverse Function Theorem. Thus, we get the existence of neighborhoods $U_2$ and $I_2$ of $0,$ and a unique differentiable function $\sigma: U_2\rightarrow I_2$ such that
	\[
	s=\sigma(r_2),\quad \sigma(0)=0,\quad \sigma'(0)=\dfrac{2T(x_0)}{r_1(m-e q_1 x_0)}>0.
	\]
	Going back through the change $s=(t-T(x_0))^2$ we get two distinct positive times $t=T(x_0)\pm\sqrt{\sigma(r_2)}$ bifurcating from $t=T(x_0).$ Since $ t_1(x_0;r_2)$ is the first return time we conclude that
	\begin{eqnarray*}
		t_1(x_0;r_2)&=&T(x_0)-\sqrt{\sigma(r_2)}\\
		&=&T(x_0)+\sqrt{\dfrac{2T(x_0)}{r_1(m-e q_1 x_0)}}\sqrt{r_2}+\CO(r_2^{3/2}).
	\end{eqnarray*}
	Finally, from \eqref{uv} we compute
	\begin{equation}\label{uvr2small}
	\begin{array}{l}
	u(x_0)=x_0+\CO(r_2),\\
	v(x_0)=r_1+\sqrt{2r_1T(x_0)(m-e q_1 x_0)}\sqrt{r_2}+\CO\left(r_2^{3/2}\right).
	\end{array}
	\end{equation}
	This concludes the proof. See Fig. \ref{saturationfig}.
\end{proof}

\begin{lemma}\label{lemma1_2}
	There exist $a,b,c,$ and $d,$ with $0<a<b<\tau$ and $0<c<d,$ such that $0<u(x_0)<\tau$ and $v(x_0)>r_1$ for every $(x_0,r_2)\in[a,b]\times[c,d].$  Moreover, for $r_2\in[c,d],$ $\mu(x_0)$ is differentiable on $[a,b]$ and $u'(x_0)^2+v'(x_0)^2\neq0$ for every $(x_0,r_2)\in[a,b]\times[c,d].$
\end{lemma}
\begin{proof}
	From \eqref{taylor_u} we have that $u(x_0)>\tau$ for $x_0$ sufficiently close to $\tau,$ and from \eqref{uvr2small} we have that for a fixed $x_0\in(0,\tau)$ there exists $r_2^*>0$ such that $u(x_0)<\tau$ for every $r_2\in(0,r_2^*].$ Therefore, for $r_2={r_2}_*<r_2^*$ there exists $x_1^*\in(0,\tau)$ such that $u(x_1^*)=\tau$ and $u(x_0)<\tau$ for $x_0<x_1^*$ sufficiently close to $x_1^*.$ Moreover, $v(x_1^*)>r_1$ and, consequently, $v(x_0)>r_1$ for $x_0<x_1^*$ sufficiently close to $x_1^*$  because $(\tau,r_1)$ is a critical point for the first integral \eqref{first}. Hence, take ${x_1}_*<x_1^*$ such that $v({x_1}_*)>r_1$ and $u({x_1}_*)<\tau.$ Thus, from the continuous dependence of the solutions on the initial conditions and parameters, we get the existence of $a,b,c,$ and $d,$ with $0<a<{x_1}_*<b<\tau$ and $0<c<{r_2}_*<d$ such that $0<u(x_0)<\tau$ and $v(x_0)>r_1$ for every $(x_0,r_2)\in[a,b]\times[c,d].$
	Furthermore, in order to get the differentiability of $\mu,$ define
	\[
	\G(t,x_0,r_2)=x(t,x_0;r_2)-x_0\exp(r_2 t).
	\]
	From the proof of Lemma \ref{lemma1}, for each $(\ov{x_0},\ov{r_2})\in[a,b]\times[c,d]$ we get  the existence of $t_1(\ov{x_0};\ov{r_2})>0$ such that $\G(t_1(\ov{x_0};\ov{r_2}),\ov{x_0},\ov{r_2})=0.$ Since
	\[
	\dfrac{\p \Gamma}{\p t}(t_1(\ov{x_0};\ov{r_2}),\ov{x_0},\ov{r_2})=(r_1-r_2-v(\ov{x_0}))u(\ov{x_0})\neq0,
	\]
 by the Implicit Function Theorem, there exists a unique differentiable function $t_2(x_0,r_2),$ defined in a neighborhood of $(\ov{x_0},\ov{r_2}),$ such that $t_2(\ov{x_0},\ov{r_2})=t_1(\ov{x_0};\ov{r_2})$ and $\G(t_2(x_0,r_2),x_0,r_2)=0$ for every $(x_0,r_2)$ in this neighborhood. According to the uniqueness property it follows that $t_1=t_2,$ which implies the differentiability of $t_1$ at $(\ov{x_0};\ov{r_2})$ and, consequently, the differentiability of $\mu$ at $x_0=\ov{x_0}$ for $r_2=\ov{r_2}.$ Since $(\ov{x_0},\ov{r_2})$ was taken arbitrarily in the compact set $[a,b]\times[c,d],$ we conclude the differentiability of $\mu$ for every $(x_0,r_2)\in[a,b]\times[c,d].$	
	Finally, notice that $F(u(x_0),v(x_0))=F(x_0,r_1-r_2).$ Assuming that $u'(x_0)=v'(x_0)=0$ and computing the derivative of the last identity in the variable $x_0$ we get that $x_0=\tau.$ Hence, we conclude that $u'(x_0)^2+v'(x_0)^2\neq0$ for every $(x_0,r_2)\in[a,b]\times[c,d].$
\end{proof}

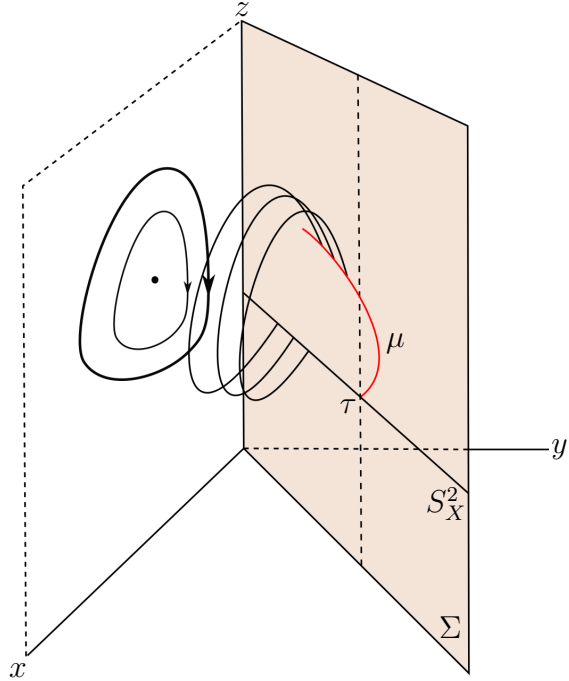
\begin{figure}[h]
	\begin{center}
		\begin{overpic}
			[width=7cm]
			{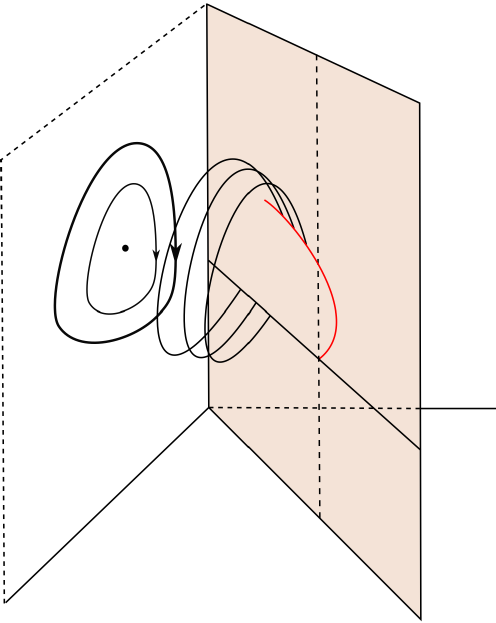}
			\put(-2,0){$x$}	\put(80,34){$y$}	\put(32,100){$z$}	\put(63,6){$\Sigma$} \put(61,25){$S^{2}_{X}$}	\put(48,40){$\tau$}	\put(55,50){$\mu$}	
		\end{overpic}
	\end{center}
	\caption{The curve $\mu$ represents the intersection between $\Sigma$ and the saturation of the curve of visible fold point of $X,$ $S^v_X,$ through the forward flow of $X.$}\label{saturationfig}
\end{figure}

\subsection{Dynamics of the sliding vector field}

In this section, we are going to look more closely at the sliding vector field. Firstly, consider Filippov system \eqref{sistema inicial sem beta} in variables $(x,w,z)=(x,x-y,z),$ where $(x,w,z)\in\R_{\geq0}\times\R\times \R_{\geq0}.$ Thus, $(\dot x,\dot w,\dot z)^T=Z(x,w,z),$ where
\begin{equation}\label{ts2}
Z(x,w,z)=\left\{\begin{array}{ll}
\left(
\begin{array}{c}
(r_1 -  z)x\\
r_2 w+x(r_1-r_2-z)\\
(e q_1 x -m)z
\end{array}
\right)&\,\, w>0,\vspace{0.2cm}\\
\left(
\begin{array}{c}
r_1 x\\
r_1 x-(x-w)\left(r_2-\dfrac{\beta_2}{\beta_1}z\right)\\
\left(\dfrac{e q_2 }{a_q}(x-w) -m\right)z
\end{array}
\right)&\,\, w<0.
\end{array}\right.
\end{equation}
In the new variables, the switching manifold is now provided by $\{(x,0,z):\,x\geq 0,\,z\geq0\}$ and the associated sliding vector field  $Z^s$ reads:
\begin{equation}\label{sliding}
\left(
\begin{array}{c}
\frac{\beta_1 r_2+\beta_2 r_1}{\beta_1+\beta_2}\, x-\frac{\beta_2}{\beta_1+\beta_2}\,x\,z\vspace{0.2cm}\\
\frac{e(a_q q_1-q_2)(r_1-r_2)\beta_1}{a_q(\beta_1+\beta_2)}\,x-m\, z+\frac{e(\beta_1 q_2+a_q \beta_2 q_1)}{a_q(\beta_1+\beta_2)}\, x\, z
\end{array}\right).
\end{equation}
The vector field \eqref{sliding} admits two equilibria, namely $(0,0)$ and
\begin{equation}\label{par2}
\big(x_c,z_c\big)=\left(\dfrac{a_q m(\beta_1 r_2+\beta_2 r_1)}{e(\beta_1 q_2 r_2+a_q \beta_2 q_1 r_1)}\,,\,r_1+\dfrac{\beta_1}{\beta_2}r_2\right)
\end{equation}
Notice that, from the original condition $\eta\in\CM,$
\[
0<x_c<\tau \text{ and } z_c>r_1=\phi-r_2.
\]
The next lemma is the main result of this subsection.
\begin{lemma}\label{lemma2}
	Let $\eta\in\CM$ and assume that
	\begin{equation}\label{condm}
	m<\dfrac{4(\beta_1+\beta_2)(r_2\beta_1+r_1\beta_2)(q_2r_2\beta_1+a_q q_1r_1\beta_2)^2}{(q_2-a_q q_1)^2(r_1-r_2)^2\beta_1^2\beta_2^2}.
	\end{equation}
	As such, the following statements hold:
	\begin{itemize}
		\item[(i)] the equilibrium $(x_c,z_c)$ is a repulsive focus;
		
		\item[(ii)] there exists $x^*\in [0,\tau)$ such that the backward orbit of $Z^s$ of any point of the straight segment $L=\{(x,\phi):\,x^*<x\leq\tau\}\subset S_X^v$ is contained in $\Sigma^s$ and converges asymptotically to the equilibrium  $(x_c,z_c).$
	\end{itemize}
\end{lemma}

\begin{remark}\label{rem1}
	If we consider the following change in parameters
	\[
	e=\dfrac{a_q m z_c}{(a_q q_1 r_1+q_2(z_c-r_1))x_c}\text{ and } \beta_2=\dfrac{r_2\beta_1}{z_c-r_1},
	\]
	then the inequality \eqref{condm} becomes
	\[
	m<\dfrac{4r_2 z_c(z_c-\phi)(a_q q_1 r_1+q_2(z_c-r_1))^2}{\phi^2(q_2-a_q q_1)^2(z_c-r_1) ^2}.
	\]
\end{remark}

\begin{proof}

	Denote by $\alpha\pm i b$ the eigenvalues of the jacobian matrix  of the vector field $Z^s$ at $(x_c,z_c).$ It is clear that \eqref{condm} implies that $b\neq0.$ In this case,
	\[
	\alpha=\dfrac{m(q_2-a_q q_1)(r_1-r_2)\beta_1\beta_2}{2(\beta_1+\beta_2)(a_q q_1 r_1 \beta_2+q_2 r_2 \beta_1)}>0.
	\]
	Hence, $(x_c,z_c)$ is a repulsive focus. This concludes the proof of statement $(i).$ In order to prove statement $(ii),$ we first claim that the sliding vector field denoted by $Z^s(x,z)=\left(Z^s_1(x,z),Z^s_2(x,z)\right)$ and provided by equation \eqref{sliding} does not admit limit cycles contained in the region $\{(x,z)\in\mathbb{R}^2:x>0,z>0\}.$ Indeed,
	\[
	\frac{\p }{\p x}\left(\frac{Z^s_1(x,z)}{x\, z}\right)+\frac{\p }{\p z}\left(\frac{Z^s_2(x,z)}{x\, z}\right)=\frac{e(q_2-a_q q_1)(r_1-r_2)\beta_1}{a_q(\beta_1+\beta_2)z^2}>0,
	\]
	for $x,z>0.$ Since the function $(x,z)\mapsto\dfrac{1}{xz}$ is $C^1$ in the region $\{(x,z)\in\mathbb{R}^2:x>0,z>0\},$ the claim follows by the Bendixson-Dulac criterion (see \cite{DumortLibreArtesQualitbook}).
	We may observe that the sliding vector field writes
	\[
	Z^s(x,z)=Z_{LV}(x,z)+\left(0,\dfrac{e(a_q q_1-q_2)(r_1-r_2)\beta_1}{a_q(\beta_1+\beta_2)}\,x\right),
	\]
	where $Z_{LV}$ is a Lotka-Volterra vector field. We know that $Z_{LV}$ admits the following first integral
	\[
	\begin{array}{rllll}
	H(x,z)&=&-m-\dfrac{r_2\beta_1+r_1\beta_2}{\beta_1+\beta_2}+\dfrac{e(a_q q_1 \beta_2+q_2\beta_1)}{a_q(\beta_1+\beta_2)}x +\vspace{0.2cm}\\
	&+&\dfrac{\beta_2}{\beta_1+\beta_2}z-m\log\left(\dfrac{e(a_q q_1 \beta_2+q_2\beta_1)}{a_qm(\beta_1+\beta_2)}\right)-\vspace{0.2cm}\\
	&-&\dfrac{r_2\beta_1+r_1\beta_2}{\beta_1+\beta_2}\log\left(\dfrac{\beta_2}{r_2\beta_1+r_1\beta_2}z\right),
	\end{array}
	\]
	that is, $\langle\nabla H(x,z),Z_{LV}(x,z)\rangle=0.$  Now, let
	\[
	a=\dfrac{(\beta_1+\beta_2)(q_2 r_2 \beta_1+a_q q_1 r_1 \beta_2)}{(q_2\beta_1+a_q q_1\beta_2)(r_2\beta_1+r_1\beta_2)}>0.
	\]
	Since $	\langle\nabla H(a x,z),Z^s(x,y)\rangle$ is equal to
	\[
\dfrac{(q_2-a_q q_1)(r_1-r_2)\beta_1(r_2\beta_1+(r_1-z)\beta_2)^2}{r_2\beta_1+r_1\beta_2}>0,
	\]
	for every $x,z>0$ and $z\neq z_c,$ we get that the level curves of $H(a x,z)$ are negatively invariant. Since $Z^s$ has no limit cycles, focus $(x_c,z_c)$ must attract the orbits of every point in the positive quadrant backwards in time.
	Finally, consider $\f(t)$ the trajectory of $Z^s$ passing through $(\phi,\tau).$ If there exists $t_s>0$ such that $\f(t_s)\in S_X^2,$ then take $x^*=\f(t_s).$ Otherwise take $x^*=0.$ It is easy to see that, in this case, $x^*<\tau.$  Indeed, $x^*\neq\tau,$ otherwise there would exist a periodic solution passing through $(\tau,r_1-r_2),$ and $\pi_2 Z^s(x,r_1-r_2)=(r_1-r_2)(e q_1 x-m)>0$ for every $x>\tau.$  Hence, the proof of statement $(ii)$ follows.
\end{proof}

We are now able to prove the main theorem of this paper.

\subsection{Proof of Theorem \ref{maintheorem}} \label{proofthm2}
Let us guarantee the existence of a \textit{Sliding  Shilnikov Connection} according to Definition \ref{defshil}. Lemma \ref{lemma1} ensures that the saturation of $S^v_X$ through the forward flow of $X$ transversally intersects the switching manifold $\Sigma$ in a curve $\mu(x_0)=(u(x_0),u(x_0),v(x_0)),$ with $0<x_0<\tau.$ Moreover, from Lemma \ref{lemma1_2} there exist $a,b,c,$ and $d,$ with $0<a<b<\tau$ and $0<c<d,$ such that $0<u(x_0)<\tau$ and $v(x_0)>r_1$ for every $(x_0, r_2)\in[a,b]\times[c,d].$ Accordingly, for some $x_0\in[a,b],$ take $(x_c,z_c)=(u(x_0),v(x_0)).$ Assume $c<r_2<d$ and
\[
m<\dfrac{4r_2 v(x_0)(v(x_0)-\phi)(a_q q_1 r_1+q_2(v(x_0)-r_1))^2}{\phi^2(q_2-a_q q_1)^2(v(x_0)-r_1) ^2}.
\]
From Lemma \ref{lemma2} and Remark \ref{rem1} we have that $(x_c,z_c)\in \Sigma^s\subset\Sigma$ is a repulsive focus of $Z^s$ and there exists $x^*\in[0,\tau)$ such that the backward orbit of any point in the straight segment $L=\{(x,\phi):\,x^*<x\leq\tau\}\subset S_X^v$ is contained in $\Sigma^s$ and converges asymptotically to $p.$
If $x^*=0,$ then from Lemma \ref{lemma2} we have characterized a sliding Shilnikov connection through the fold-regular point $(x_0,x_0,\phi)$ and the pseudo-equilibrium provided by $(u(x_0),u(x_0),v(x_0)).$
Now, assume that $x^*\neq 0.$ It remains to prove that $(x_c,z_c)=(u(x_0),v(x_0))$ implies $x^*<x_0.$ Notice that, in this case, points $(x_c,z_c)$ and $(x_0,\phi)$ lie in the same level set of $F.$ Recall that $F$ is the first integral \eqref{first} of the Lotka-Volterra system \eqref{eq projecao}. Denote $C=F^{-1}(F(x_c,z_c)).$  Firstly, we shall study the behavior of $Z^s$ on $C,$ which is equivalent to analyzing the sign of the product $\langle\nabla F(x,z), Z^s(x,z)\rangle,$ that is, the sign of
\[
\dfrac{(r_1 - r_2 - z) (e q_2 x (r_1 - z) + a_q (-e q_1 r_1 x + m z)\beta_1}{a_q z (\beta_1 + \beta_2)},
\]
for $(x,z)\in C.$ Since $a_q z (\beta_1 + \beta_2)>0$ and $r_1 - r_2 - z\leq0$ for $z\geq r_1-r_2,$ it is sufficient to analyze the sign of $e q_2 x (r_1 - z) + a_q (-e q_1 r_1 x + m z)$ on $C.$
Notice that equation
\begin{equation}\label{hiperbola}e q_2 x (r_1 - z) + a_q (-e q_1 r_1 x + m z)=0\end{equation}
describes a hyperbole containing points $(0,0)$ and $(x_c,z_c).$ Indeed, \eqref{hiperbola} is a quadratic equation of the form $\mathcal{A}x^2+\mathcal{B}xz+\mathcal{C}z^2+\mathcal{D}x+\mathcal{E}z+\mathcal{F}=0,$ where $\mathcal{B}=-e q_2,$ $\mathcal{D}=r_1 e(q_2-a_q q_1),$ $\mathcal{E}=a_q m,$ $\mathcal{A}=\mathcal{C}=\mathcal{F}=0,$ and so $\mathcal{B}^2-4 \mathcal{A}\mathcal{C}=\mathcal{B}^2>0.$ Due to convexity, each connected component of the hyperbola \eqref{hiperbola} intersects $C$ at most in two points. Solving \eqref{hiperbola}, we get
\[
z=z_h(x)=\dfrac{e(q_2-a_q q_1)r_1 x}{e q_2 x-a_q m}.
\]
Define
\[
F_c(x)=F(x,z_h(x))-F(x_c,z_c).
\]
Notice that $F_c(x)>0,$ $F_c(x)=0,$ and  $F_c(x)<0$ imply $(x,z_h(x))\in\textrm{ext}(C),$  $(x,z_h(x))\in C,$ and $(x,z_h(x))\in\textrm{int}(C),$ respectively. Clearly, $F_c(x_c)=0.$ Moreover,
\[
F'_c(x_c)=-\dfrac{e q_1 r_1(\tau-x_c)^2+\tau(r_1-x_c)^2}{r_1(\tau-x_c)x_c}<0.
\]
This implies that $F_c(x)>0$ for every $x\in(0,x_c).$ Consequently, $e q_2 x (r_1 - z) + a_q (-e q_1 r_1 x + m z)<0$ and $\langle\nabla F(x,z), Z^s(x,z)\rangle>0$ for every $(x,z)\in C$ such that $x\in(0,x_c).$ It means that the vector field $Z^s$ points outwards $C$ provided that $(x,z)\in C$ and  $x\in(0,x_c).$

Finally, let $\f(t)$ be the trajectory of $Z^s$ passing through $(\phi,\tau).$ Since $x^*=\pi_x \f(t_s)$ for some $t_s>0$ and $\f(t_s)\in S_X^v,$ there exists $t_s'\in(0,t_s)$ such that $\pi_x \f(t'_s)= x_c.$ Moreover, $\f(t'_s)\in\textrm{ext}(C)$ because $(x_c,z_c)$ is a repulsive focus lying on $C.$ From the previous comments, the trajectory $\f(t)$ remains in the exterior of $C$ for every $t\in[t_s',t_s].$ Hence, $\f(t_s)\in\textrm{ext}(C)$ implying that $x^*<x_0.$ Then, applying Lemma \ref{lemma2} we have characterized a sliding Shilnikov connection through the fold-regular point $(x_0,x_0,\phi)$ and the pseudo-equilibrium point.

Now, define $\widetilde\CN$ as  the set of parameter vectors $\eta=(r_1,r_2,a_q,q_1,q_2,$ $\beta_1,\beta_2,$ $e,m)\in\CM$ satisfying the inequalities
\begin{equation}\label{ine}
\begin{array}{l}
a\leq x_0\leq b,\,\, c\leq r_2\leq d,\,\,  \text{and} \,\, m<M(x_0,r_2), \,\, \text{where}\vspace{0.3cm}\\
M(x_0,r_2)=\dfrac{4r_2 v(x_0)(v(x_0)-\phi)(a_q q_1 r_1+q_2(v(x_0)-r_1))^2}{\phi^2(q_2-a_q q_1)^2(v(x_0)-r_1) ^2},
\end{array}
\end{equation}
and the identities
\begin{equation}\label{ident}
\begin{array}{rl}
e=&E(x_0):=\dfrac{a_q m v(x_0)}{(a_q q_1 r_1+q_2(v(x_0)-r_1))u(x_0)}\quad \text{and}\vspace{0.2cm}\\
\beta_2=&B(x_0):=\dfrac{r_2\beta_1}{v(x_0)-r_1}.
\end{array}
\end{equation}
The identities \eqref{ident} come from assuming the equality $(x_c,z_c)=(u(x_0),v(x_0))$ (see Remark \ref{rem1}). From the construction above, the Filippov system \eqref{sistema inicial} possesses a sliding Shilnikov orbit whenever $\eta\in\widetilde\CN.$

We shall identify a codimension one submanifold $\CN\subset\widetilde\CN$  of  $\CM$ as follows. Firstly, notice how  $M$ is a positive continuous function and, therefore, assumes a minimum $M_0>0$ on the compact set $[a,b]\times[c,d].$ Moreover, $E'(x_0)^2+B'(x_0)^2\neq0$ for every $x_0\in(0,\tau).$ Indeed, it is easy to see that $E'(x_0)^2+B'(x_0)^2=0$ if, and only if, $u'(x_0)^2+v'(x_0)^2=0,$ which would contradict Lemma \ref{lemma1_2}. Without loss of generality, assume that for some $x_0\in(a,b),$ $B'(x_0)\neq0.$ According to the inverse function theorem, function $B$ can be locally inverted, that is, there exists a neighborhood $\CB$ of $B(x_0)$ and a unique function $B^{-1}:\CB\rightarrow(a,b)$ such that $ B\circ B^{-1}(\beta_2)=\beta_2$ whenever $\beta_2\in\CB.$ Hence, taking
\begin{equation*}
c\leq r_2\leq d,\,\, m\leq M_0,\,\, \beta_2\in\CB\,\, \text{and}\,\, e=E\circ B^{-1}(\beta_2),
\end{equation*}
the inequalities \eqref{ine} and the identities \eqref{ident} are fulfilled. Thus, for $\eta=(r_1,r_2,a_q,q_1,$ $q_2,\beta_1,\beta_2,e,m)\in\CM,$ define $\CN\subset\widetilde{\CN}$ provided by
\[	
\CN=\{\eta\in\CM:\, c< r_2< d,\, m< M_0,\, \beta_2\in\CB\,\,\text{and}\,\, e=E\circ B^{-1}(\beta_2)\}.
\]
Notice that $\CN$ is a graph defined in a open domain. Therefore, $\CN$ is a codimension one submanifold of $\CM.$
Finally, the existence of the neighborhood $\CU\subset \CM$ of $\CN,$ satisfying that the Filippov system \eqref{sistema inicial} behaves chaotically whenever $\eta\in \CU,$ follows directly from Theorem \ref{theo:topological}.

\section{Numerical Example}\label{simul}

In order to corroborate our results, we perform a numerical simulation of the Filippov system \eqref{sistema inicial sem beta} that puts in evidence the existence of the Shilnikov sliding connection obtained in the previous section. We were able to find parameter values for which the repulsive focus $(x_c,z_c)$ of the sliding vector field $Z^s$ changes its position crossing the curve $\mu.$ Recall that the curve $\mu,$ provided by Lemma \ref{lemma1}, is the saturation of $S_X^v$ (the curve of visible fold points of $X$) through the flow of $X$ intersected with $\Sigma.$ The simulation rely on computer algebra and numerical evaluations carried out with the software application MATHEMATICA (see \cite{mathematica}), which automatically computes Filippov sliding modes.

Notice that the vector field $X$ provided by \eqref{sistema inicial sem beta} does not depend on the parameter $\beta_1.$ Thus, the repulsive focus $(x_c,z_c)$ can be moved by varying the parameter $\beta_1$ keeping the trajectories of $X$ unchanged. Accordingly, we shall fix all the parameter values but $\beta_1$ according to Table \ref{tabela1}:
\begin{table}[h]
	\begin{center}
		\vspace{0.1cm}
		\begin{tabular}{cccc}\hline
			Parameter & Value \\
			\hline  \vspace{-0.2cm}\\
			$m$ & 0,790&\\
			$r_1$&0,836&\\
			$e$ & 0,948&\\
			$q_1$ & 0,772&\\
			$a_q$ & 0,660&\\
			$q_2$ & 1,084&\\
			$\beta_2$ & 0,896&\\
			$r_2$ & 0,126\\
			\hline
		\end{tabular}
	\end{center}
	\caption{Values of the parameters of the Filippov system \eqref{sistema inicial sem beta} for the numerical analysis.}\label{tabela1}
\end{table}

Taking either $\beta_1=6$ or $\beta_1=10$ and considering the parameter values provided by Table \ref{tabela1}, we see that the conditions of Lemma \ref{lemma2} are satisfied, that is, $(x_c,z_c)$ is a repulsive focus. For $\beta_1=6,$ the numerical simulation indicates that $(x_c,z_c)$ is below curve $\mu$ (see Fig. \ref{figsimul}$(a)$). For $\beta_1=10,$ the numerical simulation indicates that $(x_c,z_c)$ is above the curve $\mu$ (see Fig. \ref{figsimul}$(c)$). Thus, from the continuous dependence on the parameter $\beta_1,$ there exists $\beta_{1}^{*},$  with $6<\beta_{1}^{*}<10,$ such that for $\beta_1=\beta_{1}^{*}$ the repulsive focus $(x_c,z_c)$ belongs to the curve $\mu.$ This gives rise to a Shilnikov sliding connection (see Fig. \ref{figsimul}$(b)$).  

\begin{figure}[h]
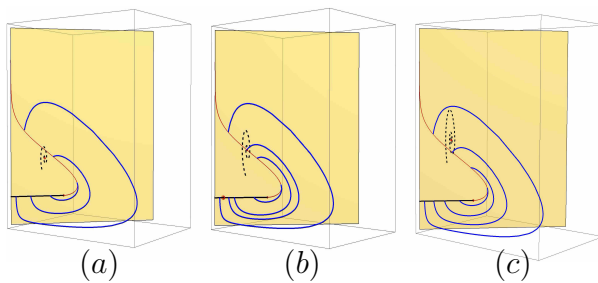

	\begin{center}
		\begin{overpic}[width=8cm]{foco2.jpg}
			\put(12,-2){$(a)$}
			\put(46,-2){$(b)$}
			\put(81,-2){$(c)$}
		\end{overpic}
	\end{center}
	
	\bigskip
	
	\caption{Relative position between curve $\mu$ and the repulsive pseudo-focus $(x_c,z_c).$ In $(a)$ and $(c),$ we are taking $\beta_1=6$ and $\beta_1=10,$ respectively.  In $(b),$ there exists $\beta_1^*$ such that $(x_c,z_c)\in\mu.$}\label{figsimul}
\end{figure}

We mention that the return $x^*$ of the sliding vector field through the point $(\tau,\phi)$ on $S_X^v$ satisfies $x^*<a$ where $0<a<\tau$ is the $x$-coordinate of the fold point which is connected to the repulsive focus through an orbit of $X.$ In other words, $a$ belongs to the segment $L$ provided by Lemma \ref{lemma2}.

In addition, we were able to check numerically that $\beta_1^*$ belongs to the interval $(7.3,8.3).$ Nevertheless, knowing the exact value of $\beta_1^*$ is not imperative in order to observe the chaotic behavior of the Filippov system \eqref{sistema inicial sem beta}. Indeed, according to Theorem \ref{maintheorem}, there exists an open interval $I,$ containing $\beta_1^*,$ such that the Filippov system \eqref{sistema inicial sem beta} behaves chaotically whenever $\beta_1\in I.$ In Figure \ref{shilfinal}, considering $\beta_1=7.8,$ we depict two trajectories of the Filippov system \eqref{sistema inicial sem beta} with close initial conditions, namely $(x_1,y_1,z_1)=(a,a,r_1-r_2)$ and  $(x_2,y_2,z_2)=(a+0.001,a+0.001,r_1-r_2),$ with $a=0.286975.$ The sensitivity to initial conditions can be observed in Figure \ref{sensitivity}.

\begin{figure}[h]
	\begin{center}
		\begin{overpic}[width=8cm]{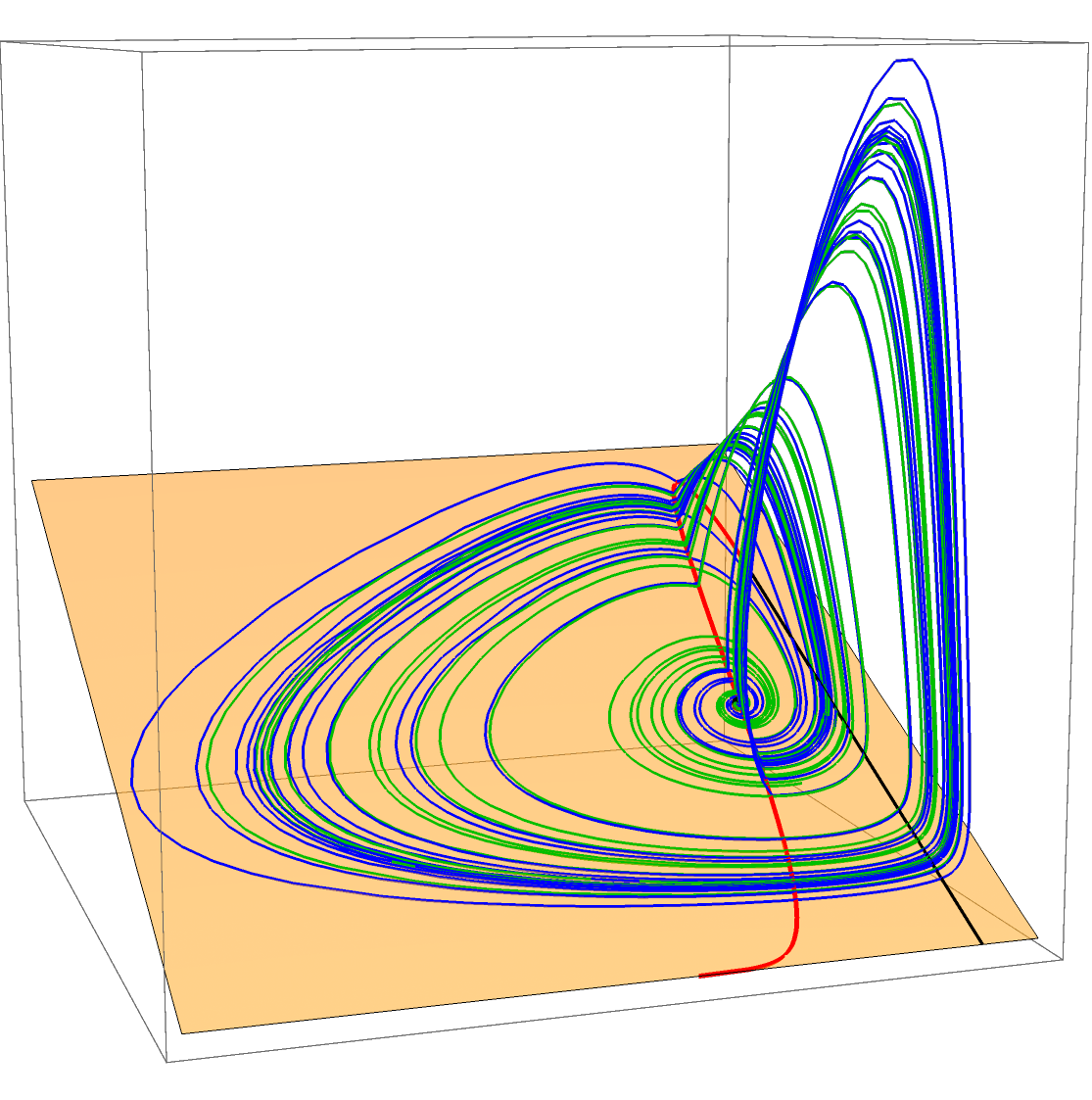}
			\put(20,10){$\Sigma^s$}
			\put(69,15){$\mu$}
			\put(82,15.5){$S_X^v$}
			\put(16,16){$x$}
			\put(8,10){$y$}
			\put(25,2){$z$}
		\end{overpic}
	\end{center}
	
	\bigskip
	
	\caption{  Two trajectories of the Filippov system \eqref{sistema inicial sem beta} with close initial conditions, namely $(x_1,y_1,z_1)=(a,a,r_1-r_2)$ and  $(x_2,y_2,z_2)=(a+0.001,a+0.001,r_1-r_2),$ with $a=0.286975.$}\label{shilfinal}
\end{figure}

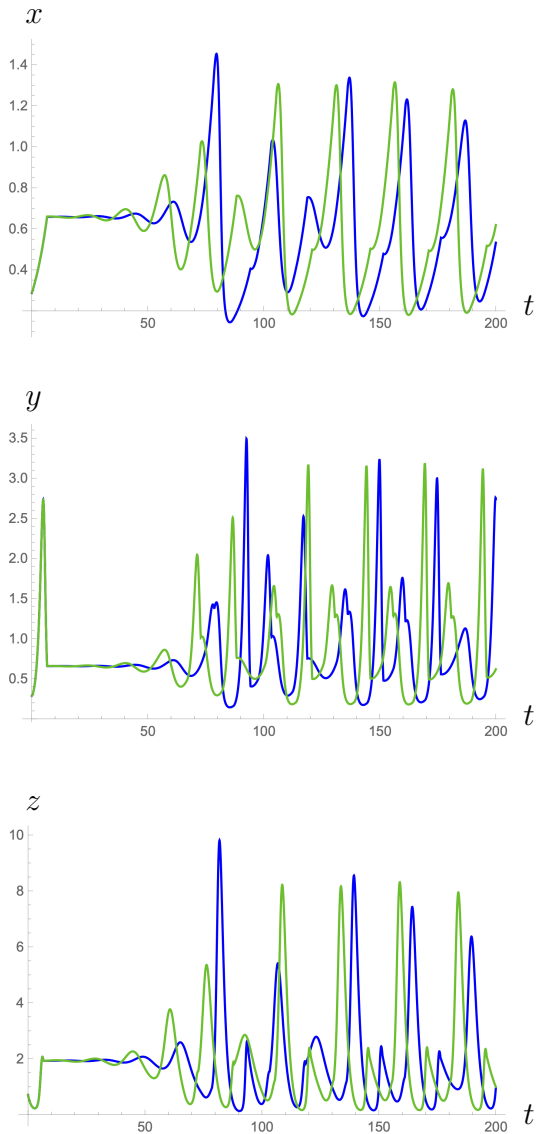
\begin{figure}[h]
	\begin{center}
		\begin{overpic}[width=7cm]{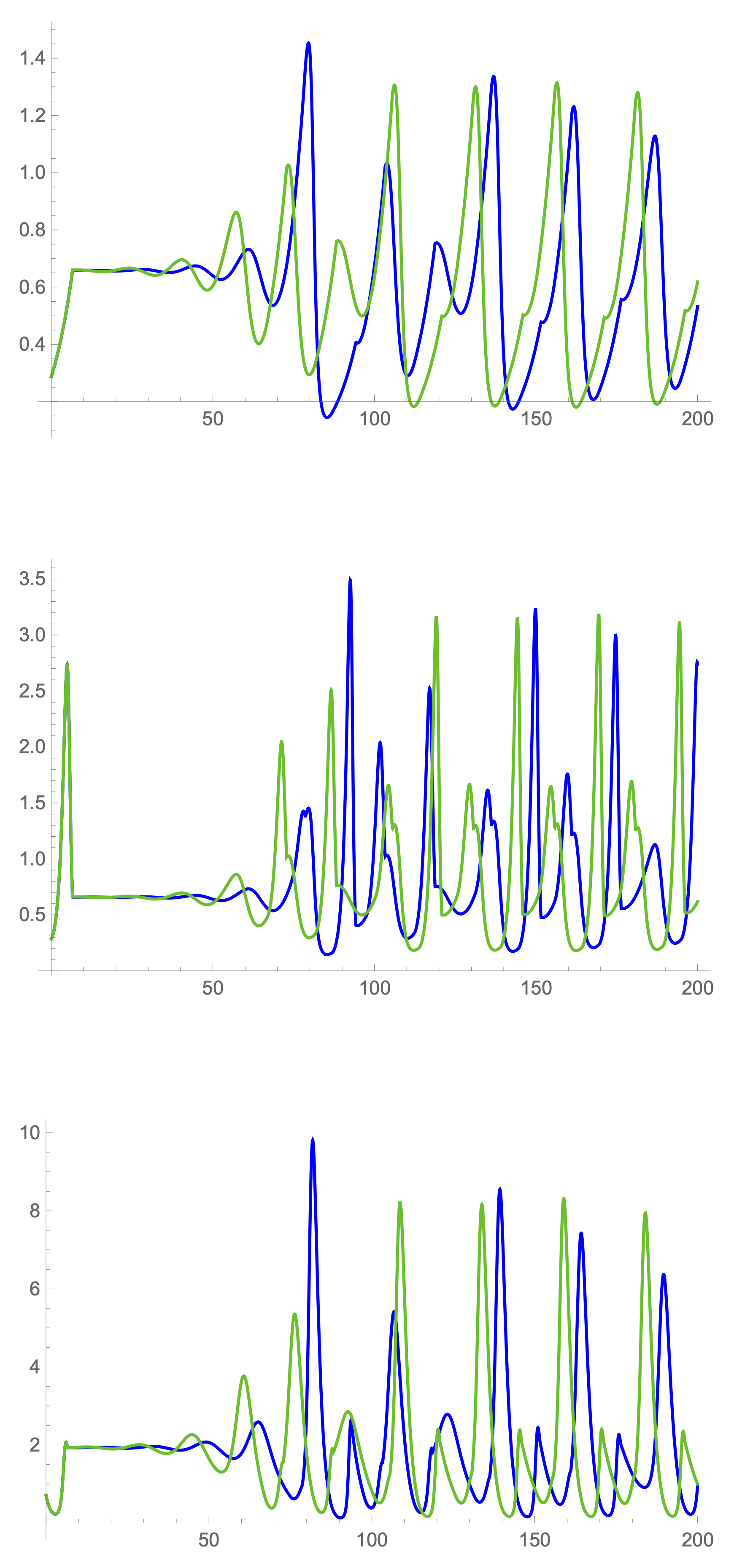}
			\put(2.7,100){$x$}
			\put(2.7,66){$y$}
			\put(2.7,30){$z$}
			\put(47,74){$t$}
			\put(47,37.5){$t$}
			\put(47,2){$t$}
		\end{overpic}
	\end{center}
	
	\bigskip
	
	\caption{  Projections onto the coordinate planes of the two trajectories of the Filippov system \eqref{sistema inicial sem beta} with close initial conditions, namely $(x_1,y_1,z_1)=(a,a,r_1-r_2)$ and  $(x_2,y_2,z_2)=(a+0.001,a+0.001,r_1-r_2),$ with $a=0.286975.$}\label{sensitivity}
\end{figure}

\section{Conclusion and Further Directions}\label{sec:conc}

In this paper, we have considered a Filippov model, introduced by \cite{Piltz}, of a 1 predator-2 prey interaction where the predator is assumed to instantaneously switch its food preference according to the availability of preys in the environment. 

In \cite{Piltz}, it has been found evidence that for a given choice of parameters such a model exhibits chaotic behavior. Here, we provided a rigorous analytic mathematical approach in order to prove that the dynamics observed in \cite{Piltz} is, in fact, chaotic.  Modern literature and new concepts were used in these terms. The main mathematical tool employed in our study was the concept  of sliding Shilnikov orbit introduced in \cite{NT}, which has been proved in \cite{NPV} to be chaotic. Our approach consisted in finding a set of parameters for which the considered model admits a sliding Shilnikov orbit. This ensures, analytically, that the model behaves chaotically for parameters taken in a neighborhood of this set. 

In the research literature, it is not so rare to find numerical analyses predicting chaotic behavior. In fact, in \cite{Kivan2006b}, Krivan and Eisner considered a 2-prey-1-predator model with Holling type II functional response and exponential prey growth. It has been numerically verified that their model exhibits a complicated behavior. We believe that this behavior can be explained by a similar mechanisms as described in the present study. 

\section*{Acknowledgements}

The authors are very grateful to Professor S\'{e}rgio Furtado dos Reis for meaningful discussion and constructive criticism on the manuscript. The authors also thank the referees for their comments and suggestions which helped us to improve the presentation of this paper. Finally, the authors thank Espa\c{c}o da Escrita - Pr\'{o}-Reitoria de Pesquisa - UNICAMP for the language services provided.

Tiago Carvalho is supported by S\~{a}o Paulo Research Foundation (FAPESP grant  2017/00883-0). Douglas Novaes is supported by S\~{a}o Paulo Research Foundation (FAPESP grants 2018/16430-8, 2018/13481-0, and 2019/10269-3) and by Conselho Nacional de Desenvolvimento Científico e Tecnológico (CNPq grants 306649/2018-7 and 438975/2018-9). Luiz F. Gon\c{c}alves is supported by the Coordena\c{c}\~{a}o de Aperfei\c{c}oamento de Pessoal de Nível Superior - Brasil (CAPES) - Finance Code 001.

%

\bibliographystyle{abbrv}  
\bibliography{references}   


\end{document}